\documentclass[12pt,leqno]{amsart}
\usepackage{amsmath}
\usepackage{amssymb}
\def\overset#1#2{{\mathrel{\mathop {{#2}_{}}\limits^{#1}}}}
\def\underset#1#2{{\mathrel{\mathop {{}_{} {#2}}\limits_{{#1}_{}}}}}
\def\upplim_#1{\underset{#1}{\overline\lim}\;}
\def\lowlim_#1{\underset{#1}{\underline\lim}\;}
\newtheorem{definition}[equation]{\indent{\it Definition}\rm }

\newtheorem{lemma}[equation]{Lemma}

\newtheorem{remark}[equation]{\indent \rm {\it Remark}}
\newtheorem{theorem}[equation]{Theorem}

\newcommand{\C}{{\mathbb{C}}}

\newcommand{\N}{\mathbf{N}}

\renewcommand{\P}{{\mathbb{P}}}

\newcommand{\R}{{\mathbb{R}}}

\newcommand{\rank}{\mathrm{rank}}

\newcommand{\supp}{\mathrm{Supp}\,}

\newcommand{\Z}{\mathbb{Z}}

\numberwithin{equation}{section}

\title[Holomorphic maps from complex discs intersecting hypersurfaces]{Holomorphic maps from complex discs intersecting hypersurfaces of projective varieties} 

\author{Si Duc Quang}

\begin{document}

\maketitle 

\begin{abstract}
In this paper, we establish second main theorems for holomorphic maps with finite growth index on complex discs intersecting families of hypersurfaces (moving and fixed) in projective varieties, where the small term is detailed estimate for various cases. Our results are generalizations and extensions of many previous second main theorems for holomorphic maps from $\C$ intersecting with hypersurfaces (moving and fixed) or moving hyperplanes.
\end{abstract}

\def\thefootnote{\empty}

\section{Introduction}

In 2020, M. Ru and N. Sibony \cite{RS} introduced the class of holomorphic maps from complex discs $\Delta (R)=\{z\in\C;|z|<R\}$ into projective spaces with finite growth index, and then they proved some second main theorems for such curves in the case where the targets are fixed values or fixed hyperplanes. Using these results, some applications on the uniqueness problem of holomorphic curves on discs have been given, such as \cite{RWa,L2}.  

Recently, in \cite{Q21} the author has established the second main theorem for the case where the targets are slowly moving hyperplanes. This work is a continuation of our studies in \cite{Q21}. In this paper we will establish second main theorems for holomorphic maps with finite growth index on $\Delta (R)$ intersecting families of hypersurfaces where the small term is detailed estimated in the following cases:
\begin{itemize}
\item The targets are arbitrary slowly moving hypersurfaces $\{Q_i\}_{i=1}^q$ in $\P^n(\C)$, the curve is algebraically nondegenerate over the field $\mathcal K_{\{Q_i\}_{i=1}^q}$ generated by the coefficients of the moving hypersurfaces, the truncation level $L_0$ for the counting functions is explicitly estimated.
\item The targets are slowly moving hypersurfaces $\{Q_i\}_{i=1}^q$ in $\P^n(\C)$ in general position, the curve may be algebraically degenerate, the truncation level for the counting function is $N=\binom{n+d}{n}-1$, where $d=lcm(\deg Q_1,\ldots,\deg Q_q)$.
\item The targets are hypersurfaces $\{Q_i\}_{i=1}^q$ in subgeneral position with respect to a projective variety $V\subset\P^n(\C)$, the curve from $\Delta (R)$ into $V$ is nondegenerate over $I_d(V)$ (i.e., the image of $f$ is not contaned in any hypersurface $Q$ of degree $d$ with $V\not\subset Q$), the truncation level for the counting function is $H_V(d)-1$, where $H_V(d)$ is the Hilbert weight of $V$.
\end{itemize}
Our results will generalize and extend many results for the case of the curves from $\C$, such as \cite{Sh90,Ru04,DT,AP,QA,Q18,Q20,Q21}. We note that if $R=+\infty$ or the growth index of the curve is zero then these generalization are trivial and the proofs will not be changed. But for the general case where $R<+\infty$ and the growth index of the curve is positive, the problem will be more interesting since the error term in the second main theorem may not be small term. Therefore, we have to carefully compute that small terms. Also, the techniques for the case of the curves from $\C$ are not enough for our purpose then we have to propose some new techniques.

In order to state the results we recall the following. Let $\nu$ be a divisor on $\Delta (R)$. We consider $\nu$ as a function on $\Delta (R)$ with values in $\Z$ such that $\supp (\nu):=\{z;\nu(z)\ne 0\}$ is a discrete subset of $\Delta (R)$. Let $k$ be a positive integer or $+\infty$. The truncated counting function of $\nu$ is defined by:
$$ n^{[k]}(t)=\sum_{|z|\le t}\min\{k,\nu(z)\} \ (0\le t\le R) \text{ and }\ N^{[k]}(r,\nu)=\int_{r_0}^{r}\dfrac{n^{[k]}(t)-n^{[k]}(0)}{t}dt.$$
Here $r_0$ is a fixed positive number so that $0<r_0<R$. We will omit the character $^{[k]}$ if $k=+\infty$.

Let $\varphi:\Delta (R)\rightarrow \C\cup\{\infty\}$ be a non-constant meromorphic function. We denote by $\nu^0_\varphi$ (resp. $\nu^\infty_\varphi$) the divisor of zeros (resp. divisor of poles) of $\varphi$ and set $\nu_\varphi=\nu^0_\varphi-\nu^\infty_\varphi$. As usual, we will write $N^{[k]}_{\varphi}(r)$ and $N^{[k]}_{1/\varphi}(r)$ for $N^{[k]}(r,\nu^0_{\varphi})$ and $N^{[k]}(r,\nu^\infty_{\varphi})$ respectively.

Let $f:\Delta (R)\rightarrow\P^n(\C)$ be a holomorphic map with a reduced representation $\tilde f=(f_0,\ldots ,f_n)$, where $f_0,\ldots,f_n$ are holomorphic functions on $\Delta (R)$ without common zero. The characteristic function of $f$ (with respect to the Fubini-Study form $\Omega_n$ on $\P^n(\C)$) is defined by
$$ T_f(r):=\int_{0}^r\frac{dt}{t}\int_{|z|<t}f^*\Omega_n.$$
According to M. Ru-N. Sibony \cite{RS}, the growth index of $f$ is defined by
$$ c_{f}=\mathrm{inf}\left\{c>0\ \biggl |\int_{0}^R\mathrm{exp}(cT_{f}(r))dr=+\infty\right\}.$$
For the convenient, we will set $c_f=+\infty$ if $\left\{c>0\ \bigl |\int_{0}^R\mathrm{exp}(cT_{f}(r))dr=+\infty\right\}=\varnothing$. 
Throughout this paper, we always assume that $c_f<+\infty$. A meromorphic function $a$ on $\Delta (R)$ (which is regarded as a holomorphic map into $\P^1(\C)$) is said to be small with respect to $f$ if $\| T_a(r)=o(T_f(r))$ as $r\rightarrow R$. Here (and throughout this paper), the notation ``$\| P$'' means the assertion $P$ holds for all $r\in (0,R)$ outside a subset $S\subset (0,R)$ with $\int_{S}e^{(1+\epsilon)(c_f+\epsilon)T_f(r)}dr<+\infty$ for some positive number $\epsilon$. Also, we will write ``$\|_\gamma P$'' if $P$ holds for all $r\in (0,R)$ outside a subset $E\subset (0,R)$ with $\int_{E}\gamma (r)dr<+\infty$, where $\gamma (r)$ is a non-negative measurable function defined on $(0,R)$ with $\int_0^R\gamma (r)dr=\infty$. We denote by $\mathcal K_f$ the set of all meromorphic functions on $\Delta (R)$ small with respect to $f$. It is easy to see that $\mathcal K_f$ is a field.

We denote by $\mathcal H$ the ring of all holomorphic functions on $\Delta (R)$. Let $Q$ be a homogeneous polynomial in $\mathcal H[x_0,\dots,x_n]$  of
degree $d \geq 1.$ Denote by $Q(z)$ the homogeneous  polynomial over $\mathbf{C}$ obtained by substituting a specific point $z \in \Delta(R)$ into the coefficients of $Q$. We also call  a moving  hypersurface in $\P^n (\mathbf{C} )$  each homogeneous polynomial $Q \in\mathcal H[x_0,\dots,x_n]$  such that all coefficients of $Q$ have no common zero. We denote by $Q(z)^*$ the support of the hypersurface $Q(z)$ for every point $z$.

Let $Q$ be a moving hypersurface in $\P^n(\mathbf{C})$ of degree $d\ge 1$ given by
$$ Q(z)=\sum_{I\in\mathcal T_d}a_I\omega^I, $$
where $\mathcal T_d=\{(i_0,\ldots,i_n)\in\N_0^{n+1}\ ;\ i_0+\cdots +i_n=d\}$, $a_I\in\mathcal H$ and $\omega^I=\omega_0^{i_0}\cdots\omega_n^{i_n}$ for $I=(i_0,\ldots,i_n)$. We consider the holomorphic map $Q':\Delta(R)\rightarrow\P^N(\mathbf{C})$, where $N=\binom{n+d}{n}$, given by
$$ Q'(z)=(a_{I_0}(z):\cdots :a_{I_N}(z))\ (\mathcal T_d=\{I_0,\ldots,I_N\}). $$
Here $I_0<\cdots<I_N$ in the lexicographic ordering. By changing the homogeneous coordinates of $\P^n(\C)$ if necessary, we may assume that for each given moving hypersurface as above, $a_{I_0}\not\equiv 0$ (note that $I_0=(0,\ldots,0,d)$ and $a_{I_0}$ is the coefficient of $\omega_n^d$). We set 
$$ \tilde Q=\sum_{j=0}^N\frac{a_{I_j}}{a_{I_0}}\omega^{I_j}. $$
We set $Q(\tilde f)=\sum_{I\in\mathcal T_d}a_If^I$ and $\tilde Q(\tilde f)=\sum_{j=0}^N\frac{a_{I_j}}{a_{I_0}}\tilde f^{I_j}$, where $\tilde f^I=f_0^{i_0}\cdots f_n^{i_n}$ for $I=(i_0,\ldots,i_n)$.

The moving hypersurface $Q$ is said to be ``slow'' (with respect to $f$) if $\|\ T_{Q'}(r)=o(T_f(r))$. This is equivalent to $\|\ T_{\frac{a_{I_j}}{a_{I_0}}}(r)=o(T_f(r))$, i.e., $\frac{a_{I_j}}{a_{I_0}}\in\mathcal K_f$, for all $j=1,\ldots,N$.

Denote by $\mathcal M$ the field of all meromorphic functions on $\Delta(R)$. Let $Q_1,\ldots,Q_q$ be $q$  moving hypersurfaces in $\P^n(\mathbf{C})$, given by $ Q_i=\sum_{I\in\mathcal T_{d_i}}a_{iI}\omega^I.$ We denote by $\mathcal K_{\{Q_i\}_{i=1}^q}$ the smallest subfield of $\mathcal M$ which contains $\mathbf{C}$ and all $\frac{a_{iI}}{a_{iJ}}$ with $a_{iJ}\not\equiv 0$. The family $\{Q_1,\ldots,Q_q\}$ is said to be in weakly $l-$subgeneral position if for every $1\le j_1<\cdots<j_{l+1}\le q,$ 
$$\bigcap_{s=1}^{l+1}Q_s(z)^*\cap V=\varnothing$$ 
for generic points $z\in\Delta(R)$ (i.e., for all $z\in\Delta(R)$ outside a discrete subset). Here, we note that $\dim\varnothing=-\infty$. 
In \cite{Q22}, we define the distributive constant of $\{Q_i\}_{i=1}^q$ by
$$ \delta:=\underset{\Gamma\subset\{1,\ldots,q\}}\max\dfrac{\sharp\Gamma}{n-\dim\left (\bigcap_{j\in\Gamma} Q_j(z)^*\right)}$$
for generic points $z\in\Delta (R)$. 
From \cite[Remark 3.7]{Q22}, we know that if $\{Q_1,\ldots,Q_q\}$ is in weakly $l-$subgeneral position then $\delta\le l-n+1$.

Our first purpose in this paper is to give a second main theorem for holomorphic maps from $\Delta(R)$ into $\P^n(\C)$ intersecting an arbitrary family of slowly moving hypersurfaces with small term being detailed estimated. Namely, we will prove the following.

\begin{theorem}\label{1.1} 
Let $f$ be a nonconstant holomorphic map of $\Delta (R)$ into $\P^n(\mathbf{C})$ with finite growth index $c_f$. Let $\{Q_i\}_{i=1}^q$ be an arbitrary family of slowly (with respect to $f$) moving hypersurfaces with $\deg Q_i = d_i\ (1\le i\le q)$ and with the distributive constant $\delta$. Assume that $f$ is algebraically nondegenerate over $\mathcal K_{\{Q_i\}_{i=1}^q}$.  Then for any $\epsilon >0$, we have
$$\|\ (q-\delta(n+1)-\epsilon)T_f(r)\le \sum_{i=1}^{q}\dfrac{1}{d_i}N^{[L_i]}_{Q_i(f)}(r)+\frac{L_0(L_0+1)}{2}(\log\gamma (r)+\epsilon\log r)+o(T_f(r)),$$
where $\gamma(r)=e^{(1+\epsilon)(c_f+\epsilon)T_f(r)}$, $L_i=\frac{1}{d_i}L_0$ and $L_0$ is a positive number which is defined by:
\begin{align*}
L_0&:=\binom{L+n}{n}p_0^{\binom{L+n}{n}\left (\binom{L+n}{n}-1\right )\binom{q}{n}-2}-1\\
\text{with }\ L&:=(n+1)d+2^{n+1}\delta(n+1)^2I(\epsilon^{-1})d,\\
d&:=lcm(d_1,\ldots,d_q) \text{(the least common multiple of all $d_i$'s)},\\
\text{ and }\ p_0&:=[\dfrac{\binom{L+n}{n}(\binom{L+n}{n}-1)\binom{q}{n}-1}{\log (1+\frac{\epsilon}{3(n+1)\delta})}]^2.
\end{align*}
\end{theorem}
Here $I(\epsilon^{-1})$ denotes the smallest integer exceed $\epsilon^{-1}$, $N^{[L]}_{Q(f)}(r)$ stands for the counting function with truncation level $L$ of the divisor of zeros of the function $Q(\tilde f)$ for any hypersurface $Q$ and a reduced representation $\tilde f$ of $f$. Note that, in this result we do not need any condition on the position of the family of moving hypersurfaces. This result is an extension of our recent second main theorem in \cite{Q22}. In this theorem, the constant $\delta (n+1)$ in front of $T_f(r)$ is called the total defect. 

Our second purpose in this paper is to establish the second main theorem with better truncation level for the counting function. We will prove the following.
\begin{theorem}\label{1.2} 
Let $f$ be a nonconstant holomorphic map of $\Delta (R)$ into $\P^n(\mathbf{C})$ with finite growth index $c_f$. Let $\{Q_i\}_{i=1}^q$ be a family of slowly (with respect to $f$) moving hypersurfaces with $\deg Q_i = d_i\ (1\le i\le q)$ in weakly general position. Set $d=lcm (d_1,\ldots ,d_q)$ and $N=\binom{n+d}{n}-1$. Assume that $Q_i(f)\not\equiv 0\ (1\le i\le q)$ and $q\ge n(N-n+2)+1$. For every $\epsilon >0$, we have 
\begin{align*}
\| \ T_f(r)\le& \dfrac{N+2}{q-n(N-n+2)+N+1}\sum_{i=1}^{q}\dfrac{1}{d_i}N^{[N]}_{Q_i(f)}(r)\\
&+\frac{N(N+1)}{2}(\log\gamma(r)+\epsilon\log r)+o(T_f(r)),\\
\end{align*}
where $\gamma(r)=e^{(1+\epsilon)(c_f+\epsilon)T_f(r)}$.
\end{theorem}

Note that, in the above result we do not need \textit{``the condition on the degeneracy of the mapping $f$''}. With this condition, we will prove a better result as follows.

\begin{theorem}\label{1.3} 
Let $f$, $\{Q_i\}_{i=1}^q$ , $d$ and $N$ be as in Theorem \ref{1.2}. Assume that $f$ is algebraically nondegenerate over $\mathcal K_{\{Q_i\}_{i=1}^q}$. For every $\epsilon >0$, we have 
$$\|\ T_f(r)\le \dfrac{N+2}{q}\sum_{i=1}^{q}\dfrac{1}{d_i}N^{[N]}_{Q_i(f)}(r)+\frac{N(N+1)}{2}(\log\gamma(r)+\epsilon\log r)+o(T_f(r)),$$
where $\gamma(r)=e^{(1+\epsilon)(c_f+\epsilon)T_f(r)}$.
\end{theorem}

Now, we consider the case where the holomorphic maps into a projective subvariety of $\P^n(\C)$. Let $V$ be a complex projective subvariety of $\mathbb P^n(\mathbb C)$ of dimension $k\ (k\le n)$. Let $Q_1,\ldots,Q_q\ (q\ge k+1)$ be $q$ hypersurfaces in $\mathbb P^n(\mathbb C)$. The family of hypersurfaces $\{Q_i\}_{i=1}^q$ is said to be in $N$-subgeneral position with respect to $V$ if for any $1\le i_1<\cdots <i_{N+1}\le q$,
$$ V\cap (\bigcap_{j=1}^{N+1}Q_{i_j}^*)=\varnothing .$$

Let $d$ be a positive integer. We denote by $I(V)$ the ideal of homogeneous polynomials in $\mathbb C [x_0,\ldots,x_n]$ defining $V$ and by $H_d$ the $\mathbb C$-vector space of all homogeneous polynomials in $\mathbb C [x_0,\ldots,x_n]$ of degree $d.$  Define 
$$I_d(V):=\dfrac{H_d}{I(V)\cap H_d}\text{ and }H_V(d):=\dim I_d(V).$$
Then $H_V(d)$ is called the Hilbert function of $V$. The element of $I_d(V)$ which is an equivalent class of a homogeneous polynomial $Q\in H_d,$ will be denoted by $[Q]$. A holomorphic map  $f:\Delta(R)\longrightarrow V$ is said to be degenerate over $I_d(V)$ if there is no $[Q]\in I_d(V)\setminus \{0\}$ such that $Q(f)\equiv 0.$ Otherwise, we say that $f$ is nondegenerate over $I_d(V)$. It is clear that if $f$ is algebraically nondegenerate, then $f$ is nondegenerate over $I_d(V)$ for every $d\ge 1.$

Our last result is stated as follows.
\begin{theorem}\label{1.4} 
Let $V$ be a projective subvariety of $\mathbb P^n(\mathbb C)$ of dimension $k\ (k\le n)$. Let $\{Q_i\}_{i=1}^q$ be hypersurfaces of $\mathbb P^n(\mathbb C)$ in $N$-subgeneral position with respect to $V$ with $\deg Q_i=d_i\ (1\le i\le q)$ and $d=lcm (d_1,\ldots,d_q)$. Let $f$ be a holomorphic map from $\Delta(R)$ into $V$ with finite growth index $c_f$ such that $f$ is nondegenerate over $I_d(V)$. Then, for every $\epsilon>0$ we have
\begin{align*}
 \|_\gamma\ &\left (q-\dfrac{(2N-k+1)H_{V}(d)}{k+1}\right )T_f(r)\\
&\le \sum_{i=1}^{q}\dfrac{1}{d_i}N^{[H_{V}(d)-1]}_{Q_i(f)}(r)+\frac{(H_{V}(d)-1)H_V(d)}{2}\left(\log\gamma +\epsilon\log r\right)+o(T_f(r)),
\end{align*}
where $\gamma(r)=e^{(1+\epsilon)(c_f+\epsilon)T_f(r)}$.
\end{theorem}

\section{Proof of Theorem \ref{1.1}}

Let $f $ be a holomorphic map from $\Delta (R)\ (0<R\le +\infty)$ into $\P^n(\C)$. From now on, we always denote by $\tilde f$ a reduced representation of $f$ and $\tilde f=(f_0,f_1,\ldots,f_n)$. We set $\|\tilde f\|=(|f_0|^2+\cdots+|f_n|^2)^{1/2}$.

In order to prove Theorem \ref{1.1}, we need the following results from \cite{CZ,DT,RS,QA} and \cite{Q22}.

\begin{theorem}[{see \cite[Theorem 4.8]{RS}}]\label{2.1}
Let $f $ be a linearly non-degenerate holomorphic map from $\Delta (R)\ (0<R\le +\infty)$ into $\P^n(\C)$. Let $\gamma (r)$ be a non-negative measurable function defined on $(0,R)$ with $\int_0^R\gamma (r)dr=\infty$ and let $H_1,\ldots,H_q$ be $q$ arbitrary hyperplanes in $\P^n(\C)$. Then, for every $\varepsilon >0$, we have
\begin{align*}
\bigl\|_\gamma &\int_0^{2\pi}\max_K\sum_{j\in K}\frac{\|\tilde f\|}{|H_j(\tilde f)|}\frac{d\theta}{2\pi}+N_W(r)\\
&\le (n+1)T_f(r)+\frac{n(n+1)}{2}(\log\gamma(r)+\epsilon\log r)+O(\log T_f(r)).
\end{align*}
\end{theorem}
Here $W$ is the generalized Wronskian of $f$, i.e., $W=\det (f_i^{(k)}; 0\le i,k\le n)$ for a reduced representation $(f_0,\ldots,f_n)$ of $f$. We also note that, in \cite[Theorem 4.8]{RS} the authors require that $f$ is of finite growth index (i.e., $c_f<+\infty$) and the obtained inequality is stated as follows
\begin{align*}
\bigl\|\ &\int_0^{2\pi}\max_K\sum_{j\in K}\frac{\|\tilde f\|}{|H_j(\tilde f)|}\frac{d\theta}{2\pi}+N_W(r)\\
&\le (n+1)T_f(r)+\frac{n(n+1)}{2}((1+\epsilon)(c_f+\epsilon)T_f(r)+\epsilon\log r)+O(\log T_f(r)).
\end{align*}
Fortunately, in their proofs, once we keep using the function $\gamma (r)$ instead of substituting $\gamma (r)=e^{(c_f+\delta)T_f(r)}$ in the proof of Lemma 4.3 and $\gamma (r)=e^{(c_f+\varepsilon)T_f(r)}$ in the proof of Theorem 4.8 then we automatically get the second main theorem in the form of Theorem \ref{2.1}.

As usual argument in the Nevanlinna theorem, from Theorem \ref{2.1} we get the following second main theorem.
\begin{theorem}[{see \cite[Theorem 1.7]{RS}}]\label{2.2}
Let $f $ be a linearly non-degenerate holomorphic map from $\Delta (R)\ (0<R\le +\infty)$ into $\P^n(\C)$. Let $\gamma (r)$ be a non-negative measurable function defined on $(0,R)$ with $\int_0^R\gamma (r)dr=\infty$ and let $H_1,\ldots,H_q$ be $q$ hyperplanes in general position in $\P^n(\C)$. Then, for every $\varepsilon >0$, we have
\begin{align*}
\bigl\|_\gamma\ (q-n-1)T_{f}(r)\le&\sum_{i=1}^qN^{[n]}_{H_i(f)}(r)+\frac{n(n+1)}{2}(1+\varepsilon)\log\gamma (r)\\
&+O(\log T_{f}(r))+\frac{n(n+1)}{2}\varepsilon\log r.
\end{align*}
\end{theorem}

\begin{lemma}[{see \cite[Lemma 2.2]{CZ}}]\label{2.3}
Let $A$ be a commutative ring and let $\{\phi_1,\ldots ,\phi_p\}$ be a regular sequence in $A$, i.e., for $i=1,\ldots ,p, \phi_i$ is not a zero divisor of $A/(\phi_1,\ldots ,\phi_{i-1})$. Denote by $I$ the ideal in $A$ generated by $\phi_1,\ldots ,\phi_p$. Suppose that for some $q,q_1,\ldots ,q_h\in A$, we have an equation
$$ \phi_1^{i_1}\cdots \phi_p^{i_p}\cdot q=\sum_{r=1}^h\phi_1^{j_1(r)}\cdots \phi_p^{j_p(r)}\cdot q_r, $$
where $(j_1(r),\ldots ,j_p(r))>(i_1,\ldots ,i_p)$ for $r=1,\ldots ,h$. Then $q\in I$.
\end{lemma}
Here, as throughout this paper, we use the lexicographic order on $\mathbb N_0^p$. Namely, 
$$(j_1,\ldots,j_p)>(i_1,\ldots,i_p)$$
iff for some $s\in\{1,\ldots ,p\}$ we have $j_l=i_l$ for $l<s$ and $j_s>i_s$.

\begin{lemma}[{see \cite[Lemma 3.2]{DT}}]\label{2.4}
Let $\{Q_i\}_{i=1}^q\ (q\ge n+1)$ be a set of homogeneous polynomials of common degree $d\ge 1$ in $\mathcal K_f[x_0,\ldots ,x_n]$ in weakly general position. Then for any pairwise different $1\le j_0,\ldots ,j_n\le q$ the sequence $\{Q_{j_0},\ldots,Q_{j_n}\}$ of elements in $\mathcal K_{\{Q_i\}}[x_0,\ldots,x_n]$ is a regular sequence, as well as all its subsequences.
\end{lemma}

Let $Q_1,\ldots,Q_q$ be $q$ moving hypersurfaces in $\P^n(\C)$. From now on, denote simply by $\mathcal K$ the field $\mathcal K_{\{Q_i\}_{i=1}^q}$. 
Denote by $\mathcal C_{\mathcal K}$ the set of all non-negative functions $h : \Delta (R)\longrightarrow [0,+\infty]$, which are of the form
$$ h=\dfrac{|g_1|+\cdots +|g_l|}{|g_{l+1}|+\cdots +|g_{l+k}|}, $$
where $k,l\in\N,\ g_1,.\ldots, g_{l+k}\in\mathcal K\setminus\{0\}$. Also, we define the set $\mathcal C_f$ similarly as $\mathcal C_{\mathcal K}$ where ${\mathcal K}$ is replaced by $\mathcal K_f$.

\begin{lemma}[{see \cite{DT}}] \label{new0}
Let $\{Q_i\}_{i=0}^n$ be a set of homogeneous polynomials of degree $d$ in $\mathcal K_f [x_0,..., x_n]$. Then there exists a function $h_1\in\mathcal C_{f}$ such that
$$ \max_{i\in\{0,...,n\}}|Q_i(\tilde f)|\le h_1\|\tilde f\|^d .$$ 
If, moreover, this set of homogeneous polynomials is in weakly general position, then there exists a nonzero function $h_2\in\mathcal C_{f}$ such that 
$$h_2\|\tilde f\|^d \le  \max_{i\in\{0,...,n\}}|Q_i(\tilde f)|.$$
 \end{lemma}

\begin{lemma}[{see \cite[Lemma 3.1]{Q22}}]\label{2.5}
Let $Q_1,\ldots,Q_{l}$ be $l$ hypersurfaces in $\P^n(\C)$ of the same degree $d\ge 1$, such that $\bigcap_{i=1}^{l}Q_i^*=\varnothing$ and
$$\dim\left (\bigcap_{i=1}^{s}Q_i^*\right )=n-u\ \forall t_{u}\le  s< t_{u+1},1\le u\le n,$$
where $t_1,t_2,\ldots,t_n$ are integers with $1=t_1<t_2<\cdots<t_{n+1}=l$. Then there exist $n+1$ hypersurfaces $P_1,\ldots,P_{n+1}$ in $\P^n(\C)$ of the forms
$$P_u=\sum_{j=1}^{t_{u}}c_{uj}Q_j, \ c_{uj}\in\C,\ u=1,\ldots,n+1,$$
such that $\bigcap_{u=1}^{n+1}P_u^*=\varnothing.$
\end{lemma}

\begin{lemma}[{see \cite[Lemma 3.2]{Q22}}]\label{2.6}
Let $t_1,t_2,\ldots,t_{n+1}$ be $n+1$ integers such that $t_1<t_2<\cdots <t_{n+1}$, and let $\delta =\underset{2\le s\le n+1}\max\dfrac{t_s-t_1}{s-1}$. 
Then for every $n$ real numbers $a_1,a_2,\ldots,a_{n}$  with $a_1\ge a_2\ge\cdots\ge a_{n}\ge 1$, we have
$$ a_1^{t_2-t_1}a_2^{t_3-t_2}\cdots a_{n}^{t_{n+1}-t_{n}}\le (a_1a_2\cdots a_{n})^{\delta}.$$
\end{lemma}

\begin{proof}[{\sc Proof of Theorem \ref{1.1}}]
Replacing $Q_i$ by $Q_i^{d/d_i}$ if necessary with the note that 
$$\dfrac{1}{d}N^{[L_0]}_{Q_i^{d/d_i}(f)}(r)\le\frac{1}{d_i}N^{[L_i]}_{Q_i(f)}(r),$$
we may assume that all hypersurfaces $Q_i\ (1\le i\le q)$ are of the same degree $d$. We may also assume that $q>\delta(n+1)$. 

Consider a reduced representation $\tilde f=(f_0,\ldots ,f_n): \Delta(R)\rightarrow \C^{n+1}$ of $f$. We also note that 
$$N^{[L_0]}_{Q_i(\tilde f)}(r)=N^{[L_0]}_{\tilde Q_i(\tilde f)}(r)+o(T_f(r)).$$
Then without loss of generality we may assume that $Q_i\in\mathcal K_f[x_0,\ldots,x_n]$.

Take $z_0$ be a point such that 
$$ \delta =\max_{\Gamma\subset\{1,\ldots,q\}}\dfrac{\sharp\Gamma}{n-\dim\bigcap_{j\in\Gamma}Q_i(z_0)^*}.$$
Note that $\delta\ge 1$, and hence $q>n+1$. If there exists $i\in\{1,\ldots,q\}$ such that $\bigcap_{\underset{j\ne i}{j=1}}^qQ_j(z_0)^*\ne\varnothing$ then 
$$ \delta\ge\dfrac{q-1}{n}>\dfrac{q}{n+1}.$$
This is a contradiction. Therefore, $\bigcap_{\underset{j\ne i}{j=1}}^qQ_j(z_0)^*=\varnothing$ for all $i\in\{1,2,\ldots,q\}$.

Since the number of slowly moving hypersurfaces occurring in this proof is finite, we may choose a function $c\in\mathcal C_{\mathcal K}$ such that for each given slowly moving hypersurface $Q$ in this proof, we have
$$ Q(z)({\bf x})\le c(z)\|{\bf x}\|^{\deg Q} $$
for all ${\bf x}=(x_0,\ldots,x_n)\in\C^{n+1}$, $z\in\Delta(R)$. 

We denote by $\mathcal I$ the set of all permutation $I=(i_1,\ldots,i_q)$ of the set $\{1,\ldots,q\}$. We set $n_0=\sharp\mathcal I$, $n_0=q!$ and write
$\mathcal I=\{I_1,\ldots,I_{n_0}\}$,
where $I_i=(I_i(1),\ldots,I_i(q))\in\mathbb N^q$ and $I_1<I_2<\cdots <I_{n_0}$ in the lexicographic order.

For each $I_i\in\mathcal I$, since $\bigcap_{j=1}^{q-1}Q_{I_i(j)}(z_0)=\varnothing$, there exist $n+1$ integers $t_{i,0},t_{i,1},\ldots,t_{i,n}$ with $0=t_{i,0}<\cdots<t_{i,n}=l_i$, where $l_i\le q-2$ such that $\bigcap_{j=0}^{l_i}Q_{I_i(j)}(z_0)^*=\varnothing$ and
$$\dim\left (\bigcap_{j=0}^{s}Q_{I_i(j)}(z_0)^*\right )=n-u\ \forall t_{i,u-1}\le s<t_{i,u},1\le u\le n.$$
Then, $\delta \ge\dfrac{t_{i,u}-t_{i,0}}{u}$ for all $1\le u\le n.$ Denote by $P'_{i,0},\ldots,P'_{i,n}$ the hypersurfaces obtained in Lemma \ref{2.5} with respect to the hypersurfaces $Q_{I_i(0)}(z_0),\ldots,Q_{I_i(l_i)}(z_0)$. Now, for each $P'_{i,j}$ constructed by
$$ P'_{i,j}=\sum_{s=0}^{t_{i,j}}a_{i,j,s}Q_{I_i(s)}(z_0)\ (a_{i,j,s}\in\C) $$
we define
$$ P_{i,j}(z)=\sum_{s=0}^{t_{i,j}}a_{i,j,s}Q_{I_i(s)}(z).$$
Hence $\{P_{i,j}\}_{j=0}^{n}$ is a family of moving hypersurfaces in $\P^n(\C)$ with $P_{i,j}(z_0)=P'_{i,j}$. Then $\bigcap_{j=0}^{n}P_{i,j}(z_0)^*=\varnothing$, and hence $\{P_{i,j}(z)\}_{j=0}^{n}$ is in weakly general position. 

We may choose a positive constant $b\ge 1$, commonly for all $I_i\in\mathcal I$, such that
$$ |P_{i,j}(z)({\bf x})|\le b\max_{0\le s\le t_{i,j}}|Q_{I_i(j)}(z)({\bf x})|, $$
for all $0\le j\le n$ and for all ${\bf x}=(x_0,\ldots,x_n)\in\C^{n+1}$ and $z\in\Delta (R)$. 
Denote by $\mathcal S$ the set of all points $z\in\Delta(R)$ such that $\bigcap_{j=0}^{n}P_{i,j}(z)\ne\varnothing$ for some $I_i$. Then $\mathcal S$ is a discrete subset of $\Delta(R)$.

By Lemma \ref{new0}, there exists a function $A\in\mathcal C_{\mathcal K}$ such that
$$ \|\tilde f (z)\|^d\le A(z)\max_{0\le j\le l_i}|Q_{I_i(j)}(z)(\tilde f(z))|\ (\forall I_i\in\mathcal I).$$
For a point $z\in \Delta\setminus\biggl\{\bigcup_{i=1}^qQ_i(\tilde f)^{-1}(\{0\})\cup \bigcup_{\underset{I_i\in\mathcal I}{0\le j\le n}}P_{i,j}(\tilde f)^{-1}(\{0\})\biggl\}$, taking $I_i\in\mathcal I$ so that
$$ |Q_{I_i(0)}(z)(\tilde f(z))|\le |Q_{I_i(1)}(z)(\tilde f(z))|\le\cdots\le |Q_{I_i(q-1)}(z)(\tilde f(z))|,$$
by Lemma \ref{3.1} we have
\begin{align*}
\prod_{i=1}^q\dfrac{\|\tilde f (z)\|^d}{|Q_i(z)(\tilde f(z))|}&\le\dfrac{A(z)^{q-l_i}}{c(z)^{l_j}}\prod_{j=0}^{l_j-1}\dfrac{c(z)\|\tilde f (z)\|^d}{|Q_{I_i(j)}(z)(\tilde f(z))|}\\
&\le\dfrac{A(z)^{q-l_j}}{c(z)^{l_j}}\prod_{j=0}^{n-1}\left(\dfrac{c(z)\|\tilde f (z)\|^d}{|Q_{I_i(t_j)}(z)(\tilde f(z))|}\right)^{t_{i,j+1}-t_{i,j}}\\
&\le \dfrac{A(z)^{q-l_j}}{c(z)^{l_j}}\prod_{j=0}^{n-1}\left(\dfrac{c(z)\|\tilde f (z)\|^d}{|Q_{I_i(t_j)}(z)(\tilde f(z))|}\right)^{\delta}\\
&\le \dfrac{A(z)^{q-l_j}b^{n\delta}}{c(z)^{l_j-n\delta}}\prod_{j=0}^{n-1}\left(\dfrac{\|\tilde f (z)\|^d}{|P_{i,j}(z)(\tilde f(z))|}\right)^{\delta}\\
&\le C(z)\prod_{j=0}^{n}\left(\dfrac{\|\tilde f (z)\|^d}{|P_{i,j}(z)(\tilde f(z))|}\right)^{\delta},
\end{align*}
where $C\in\mathcal C_{\mathcal K}$.

Now, for each non-negative integer $L$, we denote by $V_L$ the vector space (over $\mathcal K$) consisting of all homogeneous polynomials of degree $L$ in $\mathcal K[x_0,\ldots ,x_n]$ and the zero polynomial. Denote by $(P_{i,1},\ldots ,P_{i,n})$ the ideal in $\mathcal K[x_0,\ldots ,x_n]$ generated by $P_{i,1},\ldots ,P_{i,n}$.
\begin{lemma}[{see \cite[Lemma 5]{AP}, \cite[Proposition 3.3]{DT}}]\label{3.3}
Let $\{P_i\}_{i=1}^q\ (q\ge n+1)$ be a set of homogeneous polynomials of common degree $d\ge 1$ in $\mathcal K_f[x_0,\ldots ,x_n]$ in weakly general position. Then for any nonnegative integer $N$ and for any $J:=\{j_1,\ldots ,j_n\}\subset\{1,\ldots ,q\},$ the dimension of the vector space $\frac{V_L}{(P_{j_1},\ldots ,P_{j_n})\cap V_L}$ is equal to the number of $n$-tuples $(s_1,\ldots ,s_n)\in\mathbb N^n_0$ such that $s_1+\cdots +s_n\le L$ and $0\le s_1,\ldots,s_n\le d-1 $. In particular, for all $L\ge n(d-1)$, we have
$$ \dim\frac{V_L}{(P_{j_1},\ldots ,P_{j_n})\cap V_L}=d^n. $$
\end{lemma}


For each positive integer $L$ divisible by $d$ and for each $({\mathbf i})=(i_1,\ldots,i_n)\in\mathbb N^n_0$ with $\sigma({\mathbf i})=\sum_{s=1}^ni_s\le\frac{L}{d}$, we set
$$W^I_{({\mathbf i})}=\sum_{({\mathbf j})=(j_1,\ldots ,j_n)\ge ({\mathbf i})}P_{i,1}^{j_1}\cdots P_{i,n}^{j_n}\cdot V_{L-d\sigma({\mathbf j})}. $$
Hence $W^I_{(0,\ldots,0)}=V_L$ and $W^I_{({\mathbf i})}\supset W^I_{({\mathbf j})}$ if $({\mathbf i})<({\mathbf j})$ in the lexicographic ordering. Then $W^I_{({\mathbf i})}$ is a filtration of $V_L$.
  
Let $({\mathbf i})=(i_1,\ldots ,i_n),({\mathbf i}')=(i_1',\ldots ,i_n')\in \mathbb N^n_0$. Suppose that $({\mathbf i}')$ follows $({\mathbf i})$ in the lexicographic ordering. 
As (3.5) in \cite{Q18}, we have
\begin{align}\label{3.4}
\dim \dfrac{W^I_{({\mathbf i})}}{W^I_{({\mathbf i}')}}=\dim \dfrac{V_{L-d\sigma({\mathbf i})}}{(P_{i,1},\ldots ,P_{i,n})\cap V_{L-d\sigma({\mathbf i})}}.
\end{align}

Fix $L$ large enough (chosen later) and set $u=u_L:=\dim V_L=\binom{L+n}{n}$. We assume that 
$$ V_L=W^I_{({\mathbf i}_1)}\supset W^I_{({\mathbf i}_2)}\supset\cdots\supset W^I_{({\mathbf i}_K)}, $$
where $W^I_{({\mathbf i}_{s+1})}$ follows $W^I_{({\mathbf i}_s)}$ in the ordering and $({\mathbf i}_K)=(\frac{L}{d},0,\ldots ,0)$. It is easy to see that $K$ is the number of $n$-tuples $(i_1,\ldots,i_n)$ with $i_j\ge 0$ and $i_1+\cdots +i_n\le\frac{L}{d}$. Then we have
$$ K =\binom{\frac{L}{d}+n}{n}.$$
For each $k\in\{1,\ldots ,K-1\}$ we set $m^I_k=\dim \frac{W^I_{({\mathbf i}_k)}}{W^I_{({\mathbf i}_{k+1})}}$, and set $m^I_K=1$. By Lemma \ref{3.3}, $m^I_k$ does not depend on $\{P_{i,1},\ldots ,P_{i,n}\}$ and $k$, but on $\sigma({\mathbf i_k})$. Hence, we set $m_k=m^I_k$. We also note that
\begin{align}\label{3.5}
m_k=d^n
\end{align}
 for all $k$ with $L-d\sigma({\mathbf i}_k)\ge nd$ (it is equivalent to $\sigma({\mathbf i}_k)\le\dfrac{L}{d}-n$).

Using the above filtration, we may choose a basis $\{\psi^I_1,\cdots,\psi^I_u\}$ of $V_L$ such that  
$$\{\psi_{u-(m_s+\cdots +m_K)+1},\ldots ,\psi^I_u\}$$
 is a basis of $W^I_{({\mathbf i}_s)}$. 

For each integer $l \ (0\le l\le \frac{L}{d})$, we set $m(l)=m_k$, where $k$ is an index such that $\sigma({\mathbf i}_k)=l$. Since $m_k$ only depends on $\sigma({\mathbf i}_k)$, the above definition of $m(l)$ is well defined. We see that
$$ \sum_{k=1}^Km_ki_{sk}=\sum_{l=0}^{\frac{L}{d}}\sum_{k|\sigma({\mathbf i}_k)=l}m(l)i_{sk}=\sum_{l=0}^{\frac{L}{d}}m(l)\sum_{k|\sigma({\mathbf i}_k)=l}i_{sk}. $$
Here ${\mathbf i}_k=(i_{1k},\ldots,i_{nk})$. Note that, by the symmetry $(i_1,\ldots,i_n)\rightarrow (i_{\sigma (1)},\ldots ,i_{\sigma (n)})$ with $\sigma\in S(n)$,  $\sum_{k|\sigma({\mathbf i}_k)=l}i_{sk}$ does not depend on $s$. We set 
$$ a:= \sum_{k=1}^Km_ki_{sk},\ \text{ which is independent of $s$ and $I$}.$$
Hence, we have
\begin{align*}
\log\prod_{l=1}^u|\psi^I_l(\tilde f)(z)|&\le a\left (\log\prod_{j=1}^n\dfrac{|P_{i,j}(\tilde f)(z)|}{\|\tilde f(z)\|^d}\right)+uL\log \|\tilde f (z)\|+\log c_I(z),
\end{align*}
i.e.,
\begin{align*}
a\left (\log\prod_{j=1}^n\dfrac{\|\tilde f(z)\|^d}{|P_{i,j}(\tilde f)(z)|}\right)\le\log\prod_{l=1}^u\frac{\|\tilde f (z)\|^L}{|\psi^I_l(\tilde f)(z)|}+\log c_I(z),
\end{align*}

Set $c_0(z)=C(z)\prod_{I}(1+c_I^{\delta/a})\in\mathcal C_{f}$. Combining the above inequality with (\ref{2.5}), we obtain that
\begin{align}\label{3.7}
\log \prod_{i=1}^q\dfrac{\|\tilde f (z)\|^d}{|Q_i(\tilde f)(z)|}\le \frac{\delta}{a}\log\prod_{l=1}^u\dfrac{\|\tilde f (z)\|^L}{|\psi^I_l(\tilde f)(z)|}+\log c_0.
\end{align}

We now write
$$ \psi^I_l=\sum_{J\in\mathcal T_L}c^I_{lJ}x^J\in V_L,\ \ c^I_{lJ}\in\mathcal K_{\{Q_i\}}, $$
where $\mathcal T_L$ is the set of all $(n+1)$-tuples $J=(i_0,\ldots,i_n)$ with $\sum_{s=0}^nj_s=L$, $x^J=x_0^{j_0}\cdots x_n^{j_n}$  and $l\in\{1,\ldots ,u\}$. For each $l$, we fix an index $J^I_l\in J$ such that $c^I_{lJ^I_l}\not\equiv 0$. Define
$$ \mu^I_{lJ}=\dfrac{c^I_{lJ}}{c^I_{lJ^I_l}},\  J\in\mathcal T_L.$$
Set $\Phi =\{\mu^I_{lJ};I\subset\{1,\ldots ,q\},\sharp I=n, 1\le l\le M, J\in\mathcal T_L\}$. Let $B=\sharp\Phi$ and note that $1\in\Phi$. One has $B\le u\binom{q}{n}(\binom{L+n}{n}-1)=\binom{L+n}{n}(\binom{L+n}{n}-1)\binom{q}{n}$. For each positive integer $l$, we denote by $\mathcal L(\Phi (l))$ the linear span over $\C$ of the set 
$$\Phi (l)=\{\gamma_1\cdots\gamma_l;\gamma_i\in\Phi\}.$$
We see that
$$ \dim\mathcal L(\Phi(l))\le\sharp\Phi (l)\le\binom{B+l-1}{B-1}.$$
We may choose a positive integer $p$ such that
$$ p\le p_0:=\left[\dfrac{B-1}{\log (1+\frac{\epsilon}{3(n+1)\delta})}\right]^2\text{ and }\dfrac{\dim\mathcal L(\Phi (p+1))}{\dim\mathcal L(\Phi (p))}\le 1+\dfrac{\epsilon}{3(n+1)\delta}. $$
Indeed, if $\dfrac{\dim\mathcal L(\Phi (p+1))}{\dim\mathcal L(\Phi (p))}> 1+\dfrac{\epsilon}{3(n+1)\delta}$ for all $p\le p_0$, we have 
$$\dim\mathcal L(\Phi (p_0+1))\ge (1+\dfrac{\epsilon}{3(n+1)\delta})^{p_0}.$$ 
Therefore, we have
\begin{align*}
\log (1+\dfrac{\epsilon}{3(n+1)\delta})&\le\dfrac{\log \dim\mathcal L(\Phi (p_0+1))}{p_0}\le\dfrac{\log \binom{B+p_0}{B-1}}{p_0}\\ 
& =\dfrac{1}{p_0}\log \prod_{i=1}^{B-1}\dfrac{p_0+i+1}{i}<\dfrac{(B-1)\log (p_0+2)}{p_0}\\
&\le \dfrac{B-1}{\sqrt{p_0}}\le \dfrac{(B-1)\log (1+\frac{\epsilon}{3(n+1)\delta})}{B-1}\\
&=\log (1+\frac{\epsilon}{3(n+1)\delta}).
\end{align*}
This is a contradiction.

We fix a positive integer $p$ satisfying the above condition. Put $s=\dim\mathcal L(\Phi (p))$ and $t=\dim\mathcal L(\Phi (p+1))$. Let $\{b_1,\ldots,b_t\}$ be an $\C$-basis of $\mathcal L(\Phi (p+1))$ such that $\{b_1,\ldots ,b_s\}$ be a $\C$-basis of $\mathcal L(\Phi (p))$.

For each $l\in\{1,\ldots ,u\}$, we set 
$$ \tilde\psi^I_l =\sum_{J\in\mathcal T_L}\mu^I_{lJ}x^I.$$
For each $J\in\mathcal T_L$, we consider homogeneous polynomials $\phi_J(x_0,\ldots ,x_n)=x^J$. Let $F$ be a holomorphic mapping of $\Delta(R)$ into $\P^{tu-1}(\C)$ with a reduced representation $\tilde F = (hb_i\phi_J(\tilde f))_{1\le i\le t,J\in\mathcal T_L}$, where $h$ is a nonzero meromorphic function on $\Delta(R)$. We see that 
$$\|\ N_h(r)+N_{1/h}(r)=o(T_f(r)). $$
Since $f$ is algebraically nondegenerate over $\mathcal K$, $F$ is linear nondegenerate over $\C$.  We see that there exist nonzero functions $c_1,c_2\in\mathcal C_{f}$ such that 
$$c_1|h|.\|\tilde f\|^L\le \|\tilde F\|\le c_2|h|.\|\tilde f\|^L.$$

For each $l\in\{1,\ldots ,u\}, 1\le i\le s$, we consider the linear form $L^I_{il}$ in $x^J$ such that 
$$ hb_i\tilde\psi^I_l(\tilde f)=L^I_{il}(\tilde F). $$
It is clear that $\{b_i\tilde\psi^I_l(\tilde f); 1\le i\le s,1\le l\le M\}$ is linearly independent over $\C$, and so is $\{L^I_{il}(\tilde F);1\le i\le s,1\le l\le u\}$. This yields that $\{L^I_{il};1\le i\le s,1\le l\le u\}$ is linearly independent over $\C$. 

For every point $z$ which is neither zero nor pole of any $hb_i\psi^I_l(\tilde f)$, we also see that
\begin{align*}
s\log\prod_{l=1}^u\dfrac{\|\tilde f (z)\|^L}{|\psi^I_l(\tilde f)(z)|}&=\log\prod_{\overset{1\le l\le u}{1\le i\le s}}\dfrac{\|\tilde F (z)\|}{|hb_i\psi^I_l(\tilde f)(z)|}+\log c_3(z)\\
&=\log\prod_{\overset{1\le l\le u}{1\le i\le s}}\dfrac{\|\tilde F (z)\|\cdot \|L^I_{il}\|}{|L^I_{il}(\tilde F)(z)|}+\log c_4(z),
\end{align*}
where $c_3,c_4$ are nonzero functions in $\mathcal C_{f}$, not depend on $f$ and $I$, but on $\{Q_i\}_{i=1}^q$.
Combining this inequality and (\ref{3.7}), we obtain that
\begin{align}\label{3.8}
\log \prod_{i=1}^q\dfrac{\|\tilde f (z)\|^d}{|Q_i(\tilde f)(z)|}\le \frac{\delta}{sa}\left (\max_{I}\log\prod_{\overset{1\le l\le u}{1\le i\le s}}\dfrac{\|\tilde F (z)\|\cdot \|L^I_{il}\|}{|L^I_{il}(\tilde F)(z)|}+\log c_4(z)\right)+\log c_0(z),
\end{align}
for all $z$ outside a discrete subset of $\Delta(R)$.

Since $\tilde F$ is linearly nondegenerate over $\C$, then the generalized Wronskian
$$ W(hb_i\tilde\phi_J(\tilde f))=\det\left ((hb_i\tilde\phi_J(\tilde f))^{(k)}; 1\le i\le t,J\in\mathcal T_L\right)_{0\le k\le tu-1}\not\equiv 0.$$
By Theorem \ref{2.4}, we have
\begin{align}\label{3.9}
\begin{split}
\Big\|_\gamma\ \int\limits_{S(r)}\max_{I}\biggl\{\log\prod_{\overset{1\le l\le u}{1\le i\le s}}\dfrac{\|\tilde F (z)\|\cdot \|L^I_{il}\|}{|L^I_{il}(\tilde F)(z)|}\biggl\}&\le tuT_F(r)-N_{W^{\alpha}(hb_i\tilde\phi_J(\tilde f))}(r)\\
&+\dfrac{(tu-1)tu}{2}\Gamma(r)+o(T_f(r)),
\end{split}
\end{align}
where $\Gamma(r)=\log\gamma(r)+\epsilon\log r$.
Integrating both sides of (\ref{3.8}) and using (\ref{3.9}), we get
\begin{align}\label{3.10}
\begin{split}
\|\ qdT_f(r)-\sum_{i=1}^qN_{Q_i(f)}(r)\le&\frac{tu\delta}{sa}T_F(r)-\frac{\delta}{sa}N_{W^{\alpha}(hb_i\tilde\phi_J(\tilde f))}(r)\\
&+\dfrac{(tu-1)tu}{2}\Gamma(r)+o(T_f(r)).
\end{split}
\end{align}

We now estimate the quantity $\sum_{i=1}^qN_{Q_i(f)}(r)-\frac{\delta}{sa}N_{W^{\alpha}(hb_i\tilde\phi_J(\tilde f))}(r)$. 
First, we recall that
$$ Q_i(x)=\sum_{J\in\mathcal T_d}a_{iJ}x^I\in\mathcal K[x_0,\ldots ,x_n]. $$
Let $T=(\cdots ,t_{kJ},\cdots )\ (k\in\{1,\ldots ,q\}, J\in\mathcal T_d)$ be a family of variables and
$$  Q^T_i=\sum_{J\in\mathcal T_d}T_{iJ}x^J\in\mathbb Z[T,x],\ \ i=1,\ldots ,q. $$
For each element $I_i\in\mathcal I$, we denote by $\tilde R_{I_i}\in\mathbb Z[T]$ the resultant of $\{Q^T_{I_i(j)}\}_{1\le j\le l_i}$. Then there exist a positive integer $\lambda$ (common for all $I_i$) and polynomials $\tilde b_{ijs}\ (0\le s\le n, 1\le j\le l_i)$ in $\mathbb Z[T,x]$, which are zero or homogeneous in $x$ with degree of $\lambda -d$ such that 
$$ x_s^\lambda\cdot\tilde R_{I_i}=\sum_{1\le j\le l_i}\tilde b_{ijs}Q^T_{I_i(j)}\ \text{ for all }s\in\{0,\ldots ,n\},$$
and $R_{I_i}=\tilde R_{I_i}(\ldots, a_{kJ},\ldots )\not\equiv 0$. We see that $R_{I_i}\in\mathcal K_{f}$ for all $i=1,\ldots,n_0$. Set
$$b_{ijs}=\tilde b_{ijs}((\ldots ,a_{jJ},\ldots ),(x_0,\ldots ,x_n)).$$
Then we have
$$ f^\lambda_s\cdot R_{I_i}=\sum_{1\le j\le l_i}b_{ijs}(\tilde f)Q_{I_i(j)}(\tilde f)\ \text{ for all }i\in\{0,\ldots ,n\}.$$
This implies that
$$ \nu_{R_{I_i}}\ge\min_{1\le j\le l_i}\nu_{Q_{I_i(j)}(\tilde f)}+\min_{0\le s\le n, 1\le j\le l_i}\nu_{b_{ijs}(\tilde f)}.$$
We set $R=\prod_{i=1}^{n_0}R_{I_i}\in\mathcal K$. It is easy to see that
$$ \nu_{b_{ijs}(\tilde f)}\ge -C\max_{k,J}\nu^\infty_{a_{kJ}}$$
for a positive constant $C$, and the left hand side of this inequality is only depend on $\{Q_i\}$. It implies that there exists a positive constant $c$, which depends only on $\{Q_i\}$, such that
$$ \min_{1\le j\le l_i}\nu_{Q_{I_i(j)}(f)}\le \nu_{R}+c\max_{k,J}\nu^\infty_{a_{kJ}}, $$
for all $I_i\in\mathcal I$.

Fix a point $z_0\in\Delta(R)$. Without lose of generality, we may assume that
$$ \nu_{Q_1(\tilde f)}(z_0)\ge\cdots\ge\nu_{Q_q(\tilde f)}(z_0),$$
and $I_1=(1,2,\ldots,q)\in\mathcal I, l_1\le q-2$.
Now, we let $I=\{1,\ldots,l_1\}\subset\{1,\ldots ,q\}$. Then
$$ \nu_{Q_{j}(f)}(z_0)\le \nu_{R}(z_0)+c\max_{k,J}\nu^\infty_{a_{kJ}}(z_0),\ j=l_1,\ldots,q. $$
On the other hand, by Lemma \ref{2.6} we have
\begin{align*}
\sum_{i=1}^{l_1-1}(\nu_{Q_i(\tilde f)}(z_0)-\nu^{[tu-1]}_{Q_{i}(\tilde f)}(z_0))&=\sum_{i=1}^{l_1}\min\{0,\nu_{Q_i(\tilde f)}(z_0)-tu+1\}\\ 
&\le\sum_{i=1}^{n}(t_{i+1}-t_i)\min\{0,\nu_{Q_{t_i}(\tilde f)}(z_0)-tu+1\}\\
&\le\sum_{i=1}^{n}\delta\min\{0,\nu_{Q_{t_i}(\tilde f)}(z_0)-tu+1\}\\
&\le\sum_{i=1}^n\delta\min\{0,\nu_{P_{1,i}(\tilde f)}(z_0)-tu+1\}.
\end{align*}
Therefore,
\begin{align}\label{new2}
\begin{split}
\sum_{i=1}^q(\nu_{Q_i(\tilde f)}(z_0)-\nu^{[tu-1]}_{Q_{i}(\tilde f)}(z_0))\le& \sum_{i=1}^n\delta\min\{0,\nu_{P_{1,i}(\tilde f)}(z_0)-tu+1\}\\
&+(q-N)(\nu_{R}(z_0)+c\max_{k,J}\nu^{\infty}_{a_{kJ}}(z_0)).
\end{split}
\end{align} 

For $I\in\mathcal I$, take linear forms $h_{il}$ in $x^J$, $1\le l\le u, s+1\le i\le t, J\in\mathcal T_L$ such that $\{L^{I}_{il};1\le l\le u, 1\le i\le s\}\cup\{h_{il}; 1\le l\le u, s+1\le i\le t\}$ is linearly independent over $\C$. We easily see that
\begin{align}\label{3.12}
\begin{split}
\nu_{W^{\alpha}(hb_i\tilde\phi_J(\tilde f))}(z_0)&=\nu_{W^{\alpha}(L^{I}_{il}(\tilde F),\ldots,L^I_{il}(\tilde F))}(z_0)\\
&\ge\sum_{\overset{1\le l\le u}{1\le i\le s}}\left (\nu_{L^I_{il}(\tilde F)}(z_0)-\nu^{[tu-1]}_{L^I_{il}(\tilde F)}(z_0)\right )\\
&\ge \sum_{\overset{1\le l\le u}{1\le i\le s}}\left (\nu_{hb_i\tilde\psi^I_{il}(\tilde f)}(z_0)-\nu^{[tu-1]}_{hb_i\tilde\psi^I_{il}(\tilde f)}(z_0)\right )\\
&\ge \sum_{\overset{1\le l\le u}{1\le i\le s}}\left (\nu_{\tilde\psi^I_{il}(\tilde f)}(z_0)-\nu^{[tu-1]}_{\tilde\psi^I_{il}(\tilde f)}(z_0)\right )-c_0\max_{1\le i\le s}\nu^{\infty}_{hb_i}(z_0),
\end{split}
\end{align}
where $c_0$ is a positive constant, which is chosen independently of $I$, since there are only finite ordered subset $I$.
On the other hand, for each $1\le l\le u, 1\le i\le s$, since
\begin{align*}
\tilde\psi^I_{l}(\tilde f)=\frac{1}{c^I_{lJ_l^I}}\prod_{j=1}^nP_{i,j}^{i_{jk}}(\tilde f)h_l(\tilde f),
\end{align*}
where $(i_{1k},\ldots ,i_{nk})=I_k, h_l\in V_{L-d\sigma({\mathbf i}_k)}$, we have
\begin{align*}
\nu_{\tilde\psi^I_{il}(\tilde f)}(z_0)-\nu^{[tu-1]}_{\tilde\psi^I_{il}(\tilde f)}(z_0) \ge\sum_{j=1}^ni_{jk}(\nu_{P_{i,j}(\tilde f)}(z_0)-\nu^{[tu-1]}_{P_{i,j}(\tilde f)}(z_0))-c_1\max_{j,J}|\nu_{a_{jJ}}(z_0)|,
\end{align*}
where $c_1$ is a constant and the maximum is taken over all $a_{jJ}\not\equiv 0$. Summing-up both sides of the above inequalities over all $1\le i\le u,1\le l\le s$, we get
\begin{align}\label{3.15}
\begin{split}
\sum_{\overset{1\le i\le u}{1\le l\le s}}&(\nu_{\tilde\psi^I_{il}(\tilde f)}(z_0)-\nu^{[tu-1]}_{\tilde\psi^I_{il}(\tilde f)}(z_0))\\
&\ge\sum_{j=1}^ns\sum_{k=1}^Km^I_ki_{jk}(\nu_{P_{i,j}(\tilde f)}(z_0)-\nu^{[tu-1]}_{P_{i,j}(\tilde f)}(z_0))-c_{2}\max_{j,J}|\nu_{a_{jJ}}(z_0)|\\
&=as\sum_{j=1}^n(\nu_{P_{i,j}(\tilde f)}(z_0)-\nu^{[tu-1]}_{P_{i,j}(\tilde f)}(z_0))-c_{2}\max_{j,J}|\nu_{a_{jJ}}(z_0)|,
\end{split}
\end{align}
where $c_{2}$ is a constant, which depends only on $\{Q_i\}, t$ and $L$.

Combining (\ref{3.12}) and (\ref{3.15}), we get
\begin{align*}
\nu_{W^{\alpha}(hb_i\tilde\phi_J(\tilde f))}(z_0)\ge as\sum_{j=1}^n(\nu_{P_{i,j}(\tilde f)}(z_0)-\nu^{[tu-1]}_{P_{i,j}(\tilde f)}(z_0))-c_{2}\max_{j,J}\nu_{a_{jJ}}(z_0) -c_0\max_{1\le i\le s}\nu^{\infty}_{hb_i}(z_0).
\end{align*} 
Combining (\ref{3.12}) and this inequality, we obtain
\begin{align*}
\frac{\delta}{as}\nu_{W^{\alpha}(b_i\tilde\phi_J(\tilde f))}&(z_0)\ge\sum_{i=1}^q(\nu_{Q_i(\tilde f)}(z_0)-\nu^{[tu-1]}_{Q_{i}(\tilde f)}(z_0))\\
&-O(\nu_{R}(z_0)+\max_{j,J}|\nu_{a_{jJ}}(z_0)|+\max_{1\le i\le s}\nu^{\infty}_{hb_i}(z_0)).
\end{align*}
Integrating both sides of the above inequality, we obtain that
$$\|\ \frac{\delta}{as}N_{W^{\alpha}(hb_i\tilde\phi_J(\tilde f))}(r)\ge \sum_{i=1}^q(N_{Q_i(\tilde f)}(r)-N^{[tu-1]}_{Q_{i}(\tilde f)}(r))+o(T_f(r)).$$
From this inequality and (\ref{3.10}) with a note that $\|\ T_F(r)=LT_f(r)+o(T_f(r))$, we have
\begin{align}\label{3.16}
\|\ (q-\frac{tuL\delta}{ads})T_f(r)\le\sum_{i=1}^q\frac{1}{d}N^{[tu-1]}_{Q_i(f)}(r)+\frac{(tu-1)tu}{2}\Gamma(r)+o(T_f(r)).
\end{align}

We choose $L:=(n+1)d+2^{n+1}\delta(n+1)^2I(\epsilon^{-1})d$ and have some following estimate:
\begin{align*}
a&=\sum_{\sigma({\mathbf i}_k)\le\frac{L}{d}}m^I_ki_{sk}\ge \sum_{\sigma({\mathbf i}_k)\le\frac{L}{d}-n}m^I_ki_{sk}=\dfrac{d^n}{n+1}\sum_{\sigma({\mathbf i}_k)\le\frac{L}{d}-n}\sum_{s=1}^{n+1}i_{sk}\\
&=\dfrac{d^n}{n+1}\cdot \binom{\frac{L}{d}}{n}\cdot \left(\dfrac{L}{d}-n\right)=d^n\binom{\frac{L}{d}}{n+1},\\
\dfrac{uL}{da}&\le \dfrac{\binom{L+n}{n}L}{d^{n+1}\binom{\frac{L}{d}}{n+1}}=(n+1)\prod_{i=1}^n\dfrac{L+i}{L-(n-i+1)d}\\
&\le (n+1)\bigl (\dfrac{L}{L-(n+1)d}\bigl)^n=(n+1)\left (1+\dfrac{(n+1)d}{L-(n+1)d}\right )^n\\
&\le (n+1)\left (1+\dfrac{2^n(n+1)d}{L-(n+1)d}\right) \le n+1+\dfrac{\epsilon}{2\delta},\\
\dfrac{tuL}{das}&\le (1+\dfrac{\epsilon}{3(n+1)\delta})(n+1+\dfrac{\epsilon}{2\delta})\le n+1+\frac{\epsilon}{\delta}.
\end{align*}
Combining this inequality with (\ref{3.16}), we get
\begin{align}
\label{3.18}
\|\ \left (q-\delta(n+1)-\epsilon\right )T_f(r)\le \sum_{i=1}^q\frac{1}{d}N^{[tu-1]}_{Q_i(f)}(r)+\frac{(tu-1)tu}{2}\Gamma(r)+o(T_f(r)).
\end{align}
Here we note that:
\begin{itemize}
\item $p_0:=\left [\dfrac{B-1}{\log (1+\frac{\epsilon}{3(n+1)\delta})}\right ]^2\le \left [\dfrac{\binom{L+n}{n}(\binom{L+n}{n}-1)\binom{q}{n}-1}{\log (1+\frac{\epsilon}{3(n+1)\delta})}\right ]^2,$
\item $tu-1\le\binom{L+n}{n}\binom{B+p}{B-1}-1\le \binom{L+n}{n}p^{B-1}-1\le \binom{L+n}{n}p_0^{\binom{L+n}{n}(\binom{L+n}{n}-1)\binom{q}{n}-2}-1=L_0.$
\end{itemize}
By these estimates and by (\ref{3.18}), we obtain
\begin{align*}
\|\ (q-\delta(n+1)-\epsilon)T_f(r)\le \sum_{i=1}^q\frac{1}{d}N^{[L_0]}_{Q_i(f)}(r)+\frac{L_0(L_0+1)}{2}\Gamma(r)+o(T_f(r)).
\end{align*}
The theorem is proved.
\end{proof}

\section{Proof of Theorems \ref{1.2} and \ref{1.3}}
 In order to prove Theorems \ref{1.2} and \ref{1.3} we need the following.
\begin{definition}[{see \cite{Q20}}]\label{new-1}
Let $V$ be a vector space over a field $\mathcal K$. Let $\mathcal A$ be a nonempty subset of $V$. The set $\mathcal A$ is said to be non-subdegenerate if it satisfies
$$(\mathcal A_1)_{\mathcal K}\cap (\mathcal A\setminus\mathcal A_1)_{\mathcal K}\ne\{0\}\ \forall\ \varnothing\ne\mathcal A_1\subsetneq\mathcal A.$$
\end{definition}
Here, by the notation $(S)_{\mathcal K}$ we denote the linear span over $\mathcal K$ of the subset $S$ of $V$. 
\begin{lemma}[{see \cite{Q20}}]\label{new-2}
	Let $\mathcal A=\{v_1,v_2,\ldots ,v_q\}\ (q\ge 2)$ be a subset of an $\mathcal K$-vector space $V$, which is non-subdegenerate. Then there  exist subsets $I_1,\ldots ,I_k$ of $\{1,\ldots ,q\}$ such that:
	
	$\mathrm(i)$\ \ $\{v_j\ |\ j\in I_1\}$ is minimal, and $\{v_j\ |\ j\in I_i\}$ is linear independent over $\mathcal {K}\ (2\le i \le k),$ 
	
	$\mathrm(ii)$\ \ $\left(\{v_j: j\in \bigcup_{i=1}^kI_i\}\right)_{\mathcal K}=\left (\mathcal A\right )_{\mathcal K}$ and $I_i\cap I_j=\emptyset\ (i\ne j)$,
	
	$\mathrm(iii)$\ \ For each $2\le i\le k,$ there exist meromorphic functions $c_{i\alpha}\in \mathcal {K}$ $(\alpha\in\bigcup_{j=1}^{i}I_j)$  such that $c_{i\alpha}\ne 0$ for all $\alpha\in I_i$ and 
	$$\sum_{\alpha\in I_1\cup\cdots\cup I_i}c_{i\alpha} v_\alpha=0.$$ 
	Moreover, $n_1+n_2+\cdots +n_k=\rank_{\mathcal K}\mathcal A-1$, where $n_1=|I_1|-2$ and $n_t=|I_t|-1 \ (2\le t\le k).$
\end{lemma}
Here, by the notation $|S|$ we denote the cardinality of the set $S$. 

\begin{remark}\label{new5}{\rm (a) We also note that:
\begin{align*}
1+n_1&=\dim\left(\{v_j: j\in I_1\}\right)_{\mathcal K},\\ 
1+n_1+n_2&= \dim\left(\{v_j: j\in I_1\}\right)_{\mathcal K}+\dim\left(\{v_j: j\in I_2\}\right)_{\mathcal K}-1\\
&\ge \dim\left(\{v_j: j\in I_1\cup I_2\}\right)_{\mathcal K},\\
\cdots&\cdots\\
1+\sum_{i=1}^kn_k&\ge \dim\left(\{v_j: j\in I_1\cup\cdots\cup I_k\}\right)_{\mathcal K}=\rank_{\mathcal K}\mathcal A.
\end{align*}
Then the inequality $n_1+n_2+\cdots +n_k=\rank_{\mathcal K}\mathcal A-1$ follows that
$$ \dim\left(\left(\{v_j: j\in I_1\cup\cdots\cup I_s\}\right)_{\mathcal K}\cap \left(\{v_j: j\in I_{s+1}\}\right)_{\mathcal K}\right)=1\ (1\le s\le k-1).$$

(b) In the case where $\rank_{\mathcal K}\mathcal A\ge 2$: if $\sharp I_t=1$ for some $t>1$ then we remove the subset $I_t$ from the above partition and hence we may suppose that $\sharp I_t\ge 2\ \forall t\ge 2$; if $\sharp I_1=2$ then we remove an index from $I_1$ and add the another one into $I_2$ to make a new subset $I_1$, and hence we may suppose that $\sharp I_1\ge 3.$}
\end{remark} 

Developing the technique of Y. Liu in \cite{L} (also see \cite{Y}) to avoid using the lemma on logarithmic derivative, in \cite{Q21} we proved the following results, which is very essential to establish second main theorem with good truncation level for holomorphic maps on complex discs with finite growth index and slowly moving hyperplanes. 
\begin{lemma}[{cf. \cite[Lemma 2.3]{Q21}}]\label{3.1}
Let  $f:\Delta(R)\rightarrow\P^n(\C)$ be a holomorphic map, and let $\{a_i\}_{i=0}^{q-1}$ be $q$ moving hyperplanes in general position with $\rank_{\mathcal K}\{(\tilde f,\tilde a_i); 0\le i\le q-1\}=\rank_{\mathcal K}(\tilde f)$, where $\mathcal K=\mathcal K_{\{a_i\}_{i=0}^{q-1}}$. Assume that there exists a partition $\{0,\ldots,q-1\}=I_1\cup I_2\cup\cdots\cup I_l$ satisfying:\\
$\mathrm{(i)}$ \ $\sharp I_1\ge 3$, $\sharp I_t\ge 2\ (2\le t\le l)$, $\{(\tilde f,\tilde a_i)\}_{i\in I_1}$ is minimal over $\mathcal K$, $\{(\tilde f,\tilde a_i)\}_{i\in I_t}$ is linearly independent over $\mathcal {R}\ (2\le t \le l), $ \\
$\mathrm{(ii)}$ \ for any $2\le t\le l,i\in I_t,$ there exist meromorphic functions $c_i\in \mathcal {R}\setminus\{0\}$ such that 
$$\sum_{i\in I_t}c_i(\tilde f,\tilde a_i)\in \biggl(\bigcup_{j=1}^{t-1}\bigcup_{i\in I_j}(\tilde f,\tilde a_i) \biggl)_{\mathcal {R}}.$$
Then we have
$$\|_\gamma\ T_f(r)\le \sum_{i=1}^{l}\sum_{j\in I_i}N^{[n_0]}_{(\tilde f,a_j)}(r)+\frac{n_0(q-1)}{2}\left ((1+\varepsilon)\log\gamma (r)+\varepsilon\log r\right )+S(r),$$
where $S(r)= O(\log T_f(r)+\max_{0\le i \le q-1}T_{a_i}(r))$, $n_1=\sharp I_1-2$, $n_t=\sharp I_t-1$ for $t=2,\ldots,l$ and $n_0=\max_{1\le i\le l}n_i$.
\end{lemma}
Here, as usual $(\tilde f,a)$ and $(\tilde f,\tilde a)$ stand for $a(\tilde f)$ and $\tilde a(\tilde f)$ respectively (once consider $a$ as a moving hypersurfaces of degree 1). However, in the proof of \cite[Lemma 2.3]{Q21} there is a minor gap about computing the multiplicity (the second inequality in Case 2 of this proof does not hold). Then we need to modify some lines to correct this proof. For the sake of completeness, we will give the sketch of the proof to introduce the necessary notations for correcting.
\begin{proof}[{\sc Sketch Proof with Correction}]
Without loss of generality, we assume that  $I_i=\{t_{i-1}+1,\ldots ,t_{i}\}\ \ (1\le i \le l)$, where $t_0=-1.$ By the minimality over $\mathcal K$ of the set $\{(f,\tilde a_i)\}_{i\in I_1}$, it follows that there exist functions $c_i\in\mathcal K\setminus\{0\}\ (1\le i\le t_{1})$ such that 
$$(\tilde f,\tilde a_{0} )=\sum_{i=1}^{t_{1}}c_{{i}}(\tilde f, \tilde a_i).$$

We consider the holomorphic map $F^1:\Delta(R)\rightarrow\P^{\sharp I_1-2}(\C)$ given by a reduced representation 
$$F^1 =\left (\frac{c_{1}}{h_1} (\tilde f, \tilde a_1),\ldots,\frac{c_{t_1}}{h_{1}}(\tilde f, \tilde a_{t_1})\right),$$
where $h_1$ is a nonzero meromorphic function on $\Delta(R)$ with 
$$\nu_{h_1}(z)=\min_{1\le i\le t_1}\nu_{c_i(f ,\tilde a_i)}(z)\ \forall z\in\Delta (R).$$ 

For each $i\ (2\le i\le l)$, by the assumption (ii), there exists a nonzero function $P_i\in \biggl(\bigcup_{j=1}^{t_{i-1}}(\tilde f,\tilde a_j) \biggl)_{\mathcal {K}}$
such that
$$ P_i=\sum_{j=t_{i-1}+1}^{t_i}c_j(\tilde f,\tilde a_j).$$
We consider the holomorphic map $F^i:\Delta(R)\to\P^{\sharp I_i-1}$ with a reduced representation
$$ F^i=(\dfrac{c_{t_{i-1}+1}}{h_i}(\tilde f,\tilde a_{t_{i-1}+1}),\cdots,\dfrac{c_{t_{i}}}{h_i}(\tilde f,\tilde a_{t_{i}})) \text{ if } \sharp I_i>1,$$
and $$ F^i=(\dfrac{1}{h_i}P_i,\dfrac{c_{t_{i}}}{h_i}(\tilde f,\tilde a_{t_{i}})) \text{ if } \sharp I_i=1,$$
where $h_i$ is a meromorphic function on $\Delta(R)$ with 
$$\nu_{h_i}(z)=\min_{t_{i-1}+1\le j\le t_i}\nu_{c_j(f ,\tilde a_j)}(z)\ \forall z\in\Delta (R).$$ 
As \cite[Inequality (5)]{Q20} we have
\begin{align}\label{new4}
\begin{split}
T_f(r)\le&\sum_{i=1}^{l}\sum_{j\in I_i}N^{[n_0]}(r,\nu_{(\tilde f,a_j)}-\nu_i)+N(r,\nu_1-\nu)+S(r)\\
&\ \ +\frac{n_0(q-1)}{2}\left ((1+\varepsilon)\log\gamma (r)+\varepsilon\log r\right )-\sum_{i=2}^lN(r,\max\{0,\nu_{P_i/h_i}-n_0\}),
\end{split}
\end{align} 
where $\nu_i(z)=\min_{t_{i-1}+1\le j\le t_i}\nu_{(f,a_j)}(r)$ and $\nu(z)=\min_{2\le j\le t_l}\nu_{(f,a_j)}(z)$ for all $z\in\Delta (R)$.

Now, for a fixed point $z\in\Delta (R)$ which is neither a pole nor a zero of any $c_j,a_{i0}\ (\forall i,j)$, without loss of generality we may assume that $\left (\frac{c_1}{h_1}(\tilde f,\tilde a_1)\right)(z)\ne 0$ and $\left(\frac{c_{t_i+1}(\tilde f,a_{t_i+1})}{h_i}\right)(z)\ne 0\ (2\le i\le l)$. We distinguish the following two cases:

Case 1: $\nu_1(z)\le n_0$. As in the Case 1 of the proof of \cite[Lemma 2.3]{Q20}, we have
\begin{align*}
\sum_{i=1}^{l}\sum_{j\in I_i}\min\{n_0,\nu_{(\tilde f,a_j)}(z)-\nu_i(z)\}&+\nu_1(z)-\sum_{i=2}^l\max\{0,\nu_{P_i/h_i}(z)-n_0\}\\
&\le \sum_{j=0}^{q-1}\min\{n_0,\nu_{(\tilde f,a_j)}(z)\}.
\end{align*}

Case 2: $\nu_1(z)> n_0$. Let $p$ be the smallest index such that $\nu_p(z)\le n_0$. We have
\begin{align*}
&\sum_{i=1}^{l}\sum_{j\in I_i}\min\{n_0,\nu_{(\tilde f,a_j)}(z)-\nu_i(z)\}+\nu_1(z)-\sum_{i=2}^l\max\{0,\nu_{P_i/h_i}(z)-n_0\}\\ 
&\le \sum_{i=1}^{l}\sum_{j\in I_i}\min\{n_0,\nu_{(\tilde f,a_j)}(z)\}-\sum_{i=1}^p\min\{n_0,\nu_i(z)\}+\nu_1(z)\\
&\ \ -\sum_{i=2}^{p}\max\{0,\nu_{i-1}(z)-\nu_{i}(z)-n_0\}\\ 
&\le \sum_{i=1}^{l}\sum_{j\in I_i}\min\{n_0,\nu_{(\tilde f,a_j)}(z)\}-pn_0+\nu_1(z)-\sum_{i=2}^{p}(\nu_{i-1}(z)-\nu_{i}(z)-n_0)\\
&=\sum_{i=1}^{l}\sum_{j\in I_i}\min\{n_0,\nu_{(\tilde f,a_j)}(z)\}-n_0+\nu_p(z)\le \sum_{i=1}^{l}\sum_{j\in I_i}\min\{n_0,\nu_{(\tilde f,a_j)}(z)\}.
\end{align*}

Hence, we have
\begin{align*}
\sum_{i=1}^{l}\sum_{j\in I_i}N^{[n_0]}(r,\nu_{(\tilde f,a_j)}-\nu_i)&+N(r,\nu_1-\nu)+\sum_{i=2}^l(N^{[n_0]}_{P_i/h_i}(r)-N_{P_i/h_i}(r))\\
&\le \sum_{j=0}^{q-1}N^{[n_0]}(r,\nu_{(\tilde f,a_j)})+S(r).
\end{align*}
This yields that
\begin{align*}
T_f(r)\le  \sum_{i=0}^{q-1}N^{[n_0]}_{(\tilde f,a_i)}(r)+\frac{n_0(q-1)}{2}\left ((1+\varepsilon)\log\gamma (r)+\varepsilon\log r\right )+S(r).
\end{align*}
The lemma is proved.
\end{proof}

\begin{proof}[\textbf{\sc Proof of Theorem \ref{1.2}}] 
By replacing $Q_i$ with $Q_i^{d/d_i}$ if necessary, we may suppose that all $Q_i$ are of the same order $d$. Denote by $V^d$ the vector space of all homogeneous polynomials of degree $d$ in $\mathcal K[x_0,\ldots ,x_n]$ (include the zero polynomial). Set $V^d_f=\{Q(\tilde f)\ ;\ Q\in V^d\}$, which is a $\mathcal K$-vector space. It is seen that $\dim V^d_f\le \dim V^d=\binom{n+d}{n}=N+1$. 

We denote by $\mathcal I$ the set of all permutations of the set $\{1,2,\ldots,q\}$. For each element $I=(i_1,\ldots,i_q)\in\mathcal I$, we set
$$N_I=\{r\in\R^+;N^{[N]}_{Q_{i_1}(\tilde f)}(r)\le\cdots\le N^{[N]}_{Q_{i_q}(\tilde f)}(r)\}.$$

Consider an element $I=(i_1,\ldots,i_q)$ of $\mathcal I$, for instance $I=(1,\ldots ,q)$. Set $A=\{1,\ldots ,n(N-n+2)+1\}$. 
We cosider a partition $B_0,\ldots ,B_l$ of $A$ satisfying:

(1) $(Q_j(\tilde f): j\in B_i)_{\mathcal K}\cap (Q_j(\tilde f): j\in A\setminus B_i)_{\mathcal K}=\{0\}$.

(2) For each $0\le i\le l$, $\{Q_j(\tilde f); j\in B_i\}$ is non-subdegenerate.
 
\noindent
Without loss of generality, we assume that $|B_0|\ge |B_1|\ge \cdots\ge |B_l|$. 

We will show that $|B_0|\ge n+1$. Indeed, suppose contrarily that $|B_0|\le n$. We consider the following two cases.
\begin{itemize}
\item Case 1: $d=1$. Then $N=n$ and $|A|=2n+1$. Since $|A\setminus B_0|\ge n+1$, it is clear that
$$(Q_i(\tilde f); i\in B_0)_{\mathcal K}\cap (Q_i(\tilde f); i\in A\setminus B_0)_{\mathcal K}= (Q_i(\tilde f); i\in B_0)_{\mathcal K}\ne \{0\}.$$
This is a contradiction.
\item Case 2: $d\ge 2$. Consider the function $g(t)=\frac{n(N-n+2)+1}{t}+t$. We have
$$ g'(t)=1-\frac{n(N-n+2)+1}{t^2}\le \frac{n^2-n(N-n+2)-1}{n^2}<0\ \forall t\in [1;n]$$
(note that $N\ge 2n$). Then $g(t)$ is decreasing on $[1;n]$. Now, for each $1\le i\le l$ we take an index $j_i\in B_i$. Then $\{Q_{j_1},\ldots ,Q_{j_l}\}\cup \{Q_j: j\in B_0\}$ is linear independent. Therefore, 
\begin{align*}
N+1\ge l+|B_0|&\ge \frac{n(N-n+2)+1}{|B_0|}-1+|B_0|\\
&\ge \frac{n(N-n+2)+1}{n}-1+n=N+1+\frac{1}{n}.
\end{align*}
This is a contradiction. 
\end{itemize}
Then, from the above two cases we must have $|B_0|\ge n+1$.

Since $B_0$ is non-subdegenerate, we obtain subsets $I_1,\ldots ,I_k \ (I_t\cap I_s=\emptyset, \forall s\ne t)$ of $B_0$ with numbers $n_1,n_1,\ldots ,n_k$ satisfying all conclusions of Lemma \ref{3.1}. By Remark \ref{new5}(a), we may assume that $n_i\ge 1\ \forall 1\le i\le k$. It is clear that $n_i\le N\ \forall i=1,\ldots,k$.

If $\sharp (I_1\cup\cdots\cup I_{k})\le n+1$, then we set $k_0=k$.

If $\sharp (I_1\cup\cdots\cup I_{k})>n+1$, we denote by $k_0$ the smallest index such that $\left |I_1\cup\cdots\cup I_{k_0}\right |\ge n+1$. Then, we see that $\left\{Q_j; j\in I_1\cup\cdots\cup I_{k_0-1}\right\}$ is linear independent. From Remark \ref{new5}(a), we easily see that
$$\dim\left ((Q_j; j\in I_1\cup\cdots\cup I_{k_0-1})_{\mathcal K}\cap (Q_j; j\in I_{k_0})_{\mathcal K}\right )\le 1.$$
Therefore
$$\sharp(\bigcup_{t=1}^{k_0}I_t)=\dim(Q_j; j\in I_1\cup\cdots\cup I_{k_0-1})_{\mathcal K}+\dim(Q_j; j\in I_{k_0})_{\mathcal K}=p+1\le N+2.$$
Consider a holomorphic map $F:\Delta(R)\rightarrow\P^{p+1}(\C)$ with a representation $\tilde F=(Q_j(\tilde f);j\in I_1\cup\cdots\cup I_{k_0})$.
By Lemma \ref{3.1}, for every $r\in N_I$, we have
\begin{align*}
\bigl\|\ \ &dT_f(r)=T_F(r)+o(T_f(r))\le\sum_{i=n(N-n+2)-N}^{n(N-n+2)+1}N^{[N]}_{Q_j(f)}(r)+\frac{N(N+1)}{2}\Gamma(r)+o(T_f(r))\\
&\le \dfrac{N+2}{q-n(N-n+2)+N+1}\sum_{i=n(N-n+2)-N}^{q}N^{[N]}_{Q_i(f)}(r)+\frac{N(N+1)}{2}\Gamma(r)+o(T_f(r))\\
&\le\dfrac{N+2}{q-n(N-n+2)+N+1}\sum_{i=1}^{q}N^{[N]}_{Q_i(f)}(r)+\frac{N(N+1)}{2}\Gamma(r)+o(T_f(r)),
\end{align*}
where $\Gamma(r)=\log\gamma(r)+\epsilon\log r$. We see that $\bigcup_{i\in\mathcal I}N_I=\R^+$ and the above inequality holds for every $r\in N_I, I\in\mathcal I$. This yields that
$$ \bigl\|\ \ T_f(r)\le\dfrac{N+2}{q-n(N-n+2)+N+1}\sum_{i=1}^{q}N^{[N]}_{Q_i(f)}(r)+\frac{N(N+1)}{2}\Gamma(r)+o(T_f(r)).$$
The theorem  is proved.
\end{proof}

\begin{proof}[{\sc Proof of Theorem \ref{1.3}}]
By repeating the argument as in the proof of Theorem \ref{1.2}, it suffices to prove the theorem for the case where all  $Q_i$ have the same degree.

Consider arbitrary $(N+2)$  polynomials $Q_{i_1},\ldots,Q_{i_{N+2}}\ (1\le i_j\le q)$. We see that $n+1\le \dim (Q_{i_j}\ ;\ 1\le j\le N+2)_{\mathcal K}\le N+1<N+2$. Then the set $\{Q_{i_1},\ldots,Q_{i_{N+2}}\}$ is of rank at least $n+1$ and is linearly dependent over $\mathcal K$. Hence, there exists a minimal subset over $\mathcal K$, for instance that is $\{Q_{i_1},\ldots,Q_{i_{t}}\}$, of  $\{Q_{i_1},\ldots,Q_{i_{N+2}}\}$ with $n+2\le t\le N+2$. Then, there exist nonzero functions 
$c_j\ (1\le j\le t)$ in $\mathcal K$ such that
$$ c_1Q_{i_1}+\cdots +c_tQ_{i_t}=0. $$
Since $Q_{i_1},\ldots,Q_{i_t}$ are in weakly general position, $t\ge n+2.$ Denote by $F$ the meromorphic mapping of $\Delta(R)$ into $\P^{t-2}(\C)$ which has a representation $\tilde F =(c_1Q_{i_1}( f),\ldots ,c_{t-1}Q_{i_{t-1}}( f))$. Since $\{Q_{i_1},\ldots,Q_{i_{t}}\}$ is minimal, $F$ is linearly nondegenerate over $\mathbb{C}$. Denote by $H_i$ the hyperplane in $\P^{t-2}(\C)$ defined by the linear from $H_i=x_{i}$ $(0\le i\le t-2)$, and $H_t$ is the hyperplane defined by the linear form $H_t=-x_0-\cdots -x_{t-2}$. Applying the second main theorem for these fixed hyperplanes, we get
\begin{align*}
\|\ \ dT_f(r)&=T_F(r)\le\sum_{j=1}^tN^{[t-2]}_{c_jQ_{i_j}(f)}(r)+\frac{t(t+1)}{2}\Gamma (r)+o(T_f(r))\\
&\le\sum_{j=1}^{N+2}N^{[N]}_{Q_{i_j}(f)}(r)+\frac{N(N+1)}{2}\Gamma (r)+o(T_f(r)),
\end{align*}
where $\Gamma(r)=\log\gamma(r)+\epsilon\log r$. Taking summing-up of both sides of this inequality over all combinations $\{i_1,\ldots,i_{N+2}\}$ with 
$1\le i_1< \cdots<i_{N+2}\le q,$ we have
$$\|\ \ dT_f(r)\le\frac{N+2}{q}\sum_{i=1}^{q}N^{[N]}_{Q_i(f)}(r)+\frac{N(N+1)}{2}\Gamma (r)+o(T_f(r)).$$
The theorem is proved.
\end{proof}

\section{Proof of Theorem \ref{1.4}}

We need the following result on the Nochka's weights for families of hypersurfaces in subgeneral position with respect to a projective varieties, which is a generalization of the classical results of Nochka \cite{Noc83} (see \cite{No05} for detail).
\begin{lemma}[{see \cite[Lemma 3.2]{QA}}]\label{4.1}
Let $V$ be a complex projective subvariety of $\mathbb P^n(\mathbb C)$ of dimension $k\ (k\le n)$. Let $Q_1,\ldots,Q_q$ be $q\ (q>2N-k+1)$ hypersurfaces in $\mathbb P^n(\mathbb C)$ in $N$-subgeneral position with respect to $V$ of the common degree $d.$ Then there are positive rational constants $\omega_i\ (1\le i\le q)$ satisfying the following:

i) $0<\omega_i \le 1,\  \forall i\in\{1,\ldots,q\}$,

ii) Setting $\tilde \omega =\max_{j\in Q}\omega_j$, one gets
$$\sum_{j=1}^{q}\omega_j=\tilde \omega (q-2N+k-1)+k+1.$$

iii) $\dfrac{k+1}{2N-k+1}\le \tilde\omega\le\dfrac{k}{N}.$

iv) For $R\subset \{1,\ldots,q\}$ with $\sharp R = N+1$, then $\sum_{i\in R}\omega_i\le k+1$.

v) Let $E_i\ge 1\ (1\le i \le q)$ be arbitrarily given numbers. For $R\subset \{1,\ldots,q\}$ with $\sharp R = N+1$,  there is a subset $R^o\subset R$ such that $\sharp R^o=\rank \{Q_i\}_{i\in R^o}=k+1$ and 
$$\prod_{i\in R}E_i^{\omega_i}\le\prod_{i\in R^o}E_i.$$
\end{lemma}

\begin{proof}[{\sc Proof of Theorem \ref{1.4}}]
As in the proof of Theorem \ref{1.1}, we may suppose that all $Q_i\ (i=1,\ldots,q)$ do have the same degree $d$.
It is easy to see that there is a positive constant $\beta$ such that $\beta \|\tilde f\|^d\ge |Q_i(\tilde f)|$ for every $1\le i\le q.$
Set $ Q:=\{1,\ldots ,q\}$. Let $\{\omega_i\}_{i=1}^q$ be as in Lemma \ref{4.1} for the family $\{Q_i\}_{i=1}^q$.  

Take a $\mathbb C$-basis $\{[A_i]\}_{i=1}^{H_V(d)}$ of $I_d(V)$, where $A_i\in H_d$. Since $f$ is nondegenerate over $I_d(V)$, it implies that the holomorphic map $F$ with the reduced representation $\tilde F=(A_1(\tilde f),\ldots,A_{H_V(d)}(\tilde f)\}$ is linearly independent over $\mathbb C$ and $dT_f(r)=T_F(r)+O(1)$. Denote by $W$ the generalized Wronskian of $\tilde F$

Let $z$ be a fixed point. There exists $R\subset Q$ with $\sharp R=N+1$ such that $|Q_{i}(\tilde f)(z)|\le |Q_j(\tilde f)(z)|,\forall i\in R,j\not\in R$. Since $\bigcap_{i\in  R}Q_i=\varnothing$, there exists a positive constant $\alpha$ such that
$$ \alpha \|\tilde f\|^d(z)\le \max_{i\in R}|Q_i(\tilde f)(z)|. $$
Choose $R^{o}\subset R$ satisfying Lemma \ref{4.1} v) with respect to numbers $\bigl \{\dfrac{\beta \|\tilde f(z)\|^d}{|Q_i(\tilde f)(z)|}\bigl \}_{i=1}^q$.  Then, we get
$$\dfrac{\|\tilde f(z)\|^{d(\sum_{i=1}^q\omega_i)}}{|Q_1^{\omega_1}(\tilde f)(z)\cdots Q_q^{\omega_q}(\tilde f)(z)|}\le C\prod_{i\in R}\left (\dfrac{\beta\|f(z)\|^d}{|Q_i(\tilde f)(z)|}\right )^{\omega_i}\le C'\dfrac{\|\tilde F\|^{k+1}(z)}{\prod_{i\in R^o}|Q_i(\tilde f)|(z)},$$
where $C'$ are positive constants (chosen commonly for all $R$). 

Therefore, for each  $z\in \Delta(R)$, we have
\begin{align*}
\log \left (\dfrac{\|f(z)\|^{d(\sum_{i=1}^q\omega_i)}}{|Q_1^{\omega_1}(\tilde f)(z)\cdots Q_q^{\omega_q}(\tilde f)(z)|}\right )\le \log\max_{R^o}\dfrac{\|\tilde F\|^{k+1}(z)}{\prod_{i\in R^o}|Q_i(\tilde f)|(z)}+O(1),
\end{align*}
where the maximum is taken over all subsets $R^o\subset\{1,\ldots,q\}$ such that $\{Q_i(\tilde f);i\in R^0\}$ is linear independent over $\C$. Integrating both sides of the above inequality and applying Theorem \ref{2.1} with the note that $\sum_{i=1}^q\omega_i=\tilde\omega_i(q-2N+k-1)+k+1$, we get
\begin{align}\label{4.5}
\begin{split}
\|_\gamma\   &d(\tilde\omega(q-2N+k-1)+k+1)T_f(r)-\sum_{i=1}^{q}\omega_iN_{Q_i(\tilde f)}(r)\\
&\le H_V(d)T_F(r)-N_{W}(r)+\frac{(H_V(d)-1)H_V(d)}{2}(\log\gamma(r)+\epsilon\log r)+o(T_f(r)).
\end{split}
\end{align}
As usual argument in Nevanlinna theory, we have
$$\sum_{i=1}^q\omega_iN_{Q_i(\tilde f)}(r)-N_{W}(r)\le \sum_{i=1}^q\omega_iN^{[H_V(d)-1]}_{Q_i(\tilde f)}(r).$$
Combining this inequality and (\ref{4.5}), we obtain
\begin{align*}
\|_\gamma\   &d(\tilde\omega(q-2N+k-1)+k+1-H_V(d))T_f(r)\\
&\le \sum_{i=1}^q\omega_iN^{[H_V(d)-1]}_{Q_i(\tilde f)}(r)+\frac{(H_V(d)-1)H_V(d)}{2}(\log\gamma(r)+\epsilon\log r)+o(T_f(r)).
\end{align*}
Since $\dfrac{k+1}{2N-k+1}\le \tilde\omega\le\frac{k}{N}$, the above inequality  implies that
\begin{align*}
\bigl\|_\gamma\  &\left(q-\frac{(2N-k+1)H_V(d)}{k+1}\right)T_f(r)\\
&\le \sum_{i=1}^q\dfrac{1}{d}N^{[H_V(d)-1]}_{Q_i(f)}(r)+\frac{(H_V(d)-1)H_V(d)N}{2dk}(\log\gamma(r)+\epsilon\log r)+o(T_f(r)).
\end{align*}
The theorem is proved.
\end{proof}

\section*{Disclosure statement}
No potential conflict of interest was reported by the author(s).

\vskip0.2cm
{\footnotesize 
\noindent
{\sc Si Duc Quang}
\vskip0.05cm
\noindent
$^1$Department of Mathematics, Hanoi National University of Education,\\
136-Xuan Thuy, Cau Giay, Hanoi, Vietnam.
\vskip0.05cm
\noindent
$^2$Thang Long Institute of Mathematics and Applied Sciences,\\
Nghiem Xuan Yem, Hoang Mai, HaNoi, Vietnam.
\vskip0.05cm
\noindent
\textit{E-mail}: quangsd@hnue.edu.vn

\end{document}